\documentclass{amsart}

\usepackage[english]{babel}
\usepackage{amsmath, amsfonts, amssymb, amsthm, faktor}

\usepackage[all]{xy}
\usepackage{tikz}
\usetikzlibrary{math}
\usetikzlibrary{patterns}
\usepackage{caption}
\usepackage{subfig}
\usepackage{float}

\usepackage{caption}
\usetikzlibrary{arrows,chains,matrix,positioning,scopes,decorations.pathreplacing,
decorations.pathmorphing,decorations.markings,arrows.meta}

\usepackage{aliascnt}
\usepackage[colorlinks, linktocpage, allcolors=black,breaklinks]{hyperref}

\usepackage{enumerate}

\usepackage{verbatim}
\usepackage[T1]{fontenc}
\usepackage{lmodern}
\usepackage{amsmath,amssymb,amsthm,mathtools}
\usepackage{booktabs,longtable,array,tabularx}
\usepackage{enumitem}
\usepackage{fancyvrb}
\usepackage{multicol}
\usepackage[margin=1in]{geometry}
\usepackage{microtype}
\allowdisplaybreaks

\theoremstyle{definition}
\newtheorem{theorem}{Theorem}[section]

\newtheorem{proposition}[theorem]{Proposition}
\newtheorem{corollary}[theorem]{Corollary}
\newtheorem{lemma}[theorem]{Lemma}
\newtheorem{question}[theorem]{Question}

\newtheorem*{acknowledgments}{Acknowledgments}
\newcommand{\N}{\mathbb N}
\newcommand{\Z}{\mathbb Z}

\newcommand{\PhiA}{\Phi}
\newcommand{\asdimB}{\operatorname{asdim}_{\mathrm B}}
\newcommand{\dist}{\operatorname{dist}}

\newcommand{\abs}[1]{\lvert #1\rvert}

\begin{document}
\title[Commuting Functions]
{Borel graphs generated by finitely many bounded-to-1 commuting Borel functions and hyperfinite equivalence relations}

\author[Ruijun Wang]{Ruijun Wang}
\address{School of Mathematical Sciences and School of Pre-university, Dalian Minzu University}
\email{wangruijun@dlnu.edu.cn}
\thanks{}



\date{}

\maketitle
\begin{abstract}
We show that a Borel graph generated by finitely many bounded-to-1 commuting functions is hyperfinite. This recovers a theorem in the forthcoming paper by Naryshkin, Shinko, Weilacher and Yu.
\end{abstract}

\maketitle

\section{Introduction}
In the seminal paper \cite{KST1999}, Kechris, Solecki and Todorcevic initiated the study of  descriptive combinatorics. In \cite{GJ2015}, Gao and Jackson developed the rectangular partition method for Schreier graphs of countable abelian groups actions, and they proved that the graph is hyperfinite. In \cite{SchneiderSeward2024}, Schneider and Seward extended this method and proved the countable Borel equivalence relation by action of locally nilpotent group is hyperfinite. In \cite{CJMST}, Conley, Jackson, Marks, Seward and Tucker-Drob also followed this idea and developed Borel asymptotic dimension of locally finite graphs, and they proved hyperfiniteness for polycyclic group actions. In \cite{BY}, Bernshteyn and Yu proved the hyperfiniteness of Borel graphs of polynomial growth using Borel asymptotic dimension.

In this paper, we study the hyperfiniteness of graphs generated by commuting functions. Let $f_i:X\to X,\ i\in I$ be a family of Borel functions and $f_i(f_j(x))=f_j(f_i(x))$ for all $i,j\in I$, and let $G$ be the Borel graph generated by the functions, $(x,y)$ is an edge iff $f_i(x)=y$ or $f_i(y)=x$ for some $i\in I$ and $x\neq y$. What is the complexity of $G$? This open question is very popular folklore and there are different versions of it, see \cite[Question 6.2]{GH} and \cite[Problem 44]{OpenProblems}. It is so popular that at least two different groups of people work on it. While their paper is still in preparation, in this paper we recover one of their theorems, Corollary \ref{main}.

\begin{theorem}[Naryshkin, Shinko, Weilacher, Yu, \cite{NSWY}]
    Let $G$ be the Borel graph generated by finitely many bounded-to-1 commuting functions. And let
    $$F=\{x:\{f_1^{a_1}f_2^{a_2}\cdots f_d^{a_d}(x):(a_1,\cdots,a_d)\in \mathbb N^d\}\text{ are pairwise distinct}\}$$
    be the free part. Then $G\upharpoonright F$ is of finite Borel asymptotic dimension and thus hyperfinite.
\end{theorem}


\begin{theorem}[Naryshkin, Shinko, Weilacher, Yu, \cite{NSWY}]
    Let $G$ be the Borel graph generated by finitely many bounded-to-1 commuting functions. And let $F$ be as above and $N$ be the complement of $F$. Then $G\upharpoonright N$ is hyperfinite.
\end{theorem}

\begin{corollary}\label{main}
    Let $G$ be the Borel graph generated by finitely many bounded-to-1 commuting functions. Then $G$ is hyperfinite.
\end{corollary}

\begin{question}
    Let $G$ be the Borel graph generated by finitely many bounded-to-1 commuting functions. Then is $G\upharpoonright N$ of finite Borel asymptotic dimension?
\end{question}

\cite[Question 6.2]{GH} remains open, however, \cite[Problem 44]{OpenProblems} is negative. There is a non-hyperfinite Borel digraph with constant forward out-neighborhood growth. Let $G$ be the undirected Schreier graph $F(2^{\mathbb F_2})$, and let $H$ be the digraph with $V(H)=V(G)\times 2$, $$(x,0)\to (y,1)\in A(H)\Longleftrightarrow (x,y)\in E(G)\text{ or }x=y,$$ and there are no more arcs. In this digraph, each vertex is of bounded degree and any directed path is of length 1, so both forward and backward neighborhood are of polynomial growth with degree 0. The connectedness relation contains a universal treeable equivalence relation and thus is not hyperfinite.

\textbf{Statement of independence of this paper.} The author was told that Naryshkin, Shinko, Weilacher and Yu proved Theorem 1.1 and Theorem 1.2 but cannot answer Question 1.4. The author finishes the proof of Theorem 1.2 independently by reducing it to Theorem 1.1. By the time of submission of this paper, their paper is still in preparation.

This paper is a collaboration with artificial intelligence. The AI model is gpt 5.6. The AI fetches papers and explains papers for human and works under human suggestions. All AI generated contents are verified by human.

The rest of the paper is organized as follows. In Section 2, we introduce some definitions and notations. In Section 3, we include a proof of Theorem 1.1 by Naryshkin, Shinko, Weilacher and Yu for completeness, and we upgrade it into Theorem 3.1 for later use. In Section 4, we show a proof of Theorem 1.2, the only use of Theorem 3.1 in Section 4 is the proof of Proposition \ref{prop:coreasdim}.

\section{Preliminaries}

A countable Borel equivalence relation $E$ is \emph{hyperfinite} if
there are finite Borel equivalence relations
\[
E_0\subseteq E_1\subseteq\cdots
\quad\text{with}\quad
E=\bigcup_n E_n.
\]
A Borel graph or digraph is hyperfinite if its connectedness relation is hyperfinite.

Let $(X,\rho)$ be a Borel extended metric space whose finite-distance
relation is countable.  One has
$\asdimB(X,\rho)\leq n$ if for every $r>0$ there is a Borel
equivalence relation $F$ such that:
\begin{enumerate}[label=(\roman*)]
\item the $F$-classes have uniformly bounded $\rho$-diameter;
\item every $\rho$-ball of radius $r$ meets at most $n+1$
$F$-classes.
\end{enumerate}
Usually, the path distance of a Borel locally finite graph is a Borel extended metric. The Borel asymptotic dimension of a graph is that of its path distance. If a Borel graph is of finite Borel asymptotic dimension then it is hyperfinite.

We will use the increasing union theorem for graphs of finite Borel asymptotic dimension.

\begin{theorem}[Conley, Jackson, Marks, Seward, Tucker-Drob, \cite{CJMST}]
\label{thm:union}
Let
\[
G_0\subseteq G_1\subseteq\cdots
\]
be locally finite Borel graphs.
Suppose $\asdimB(G_n)<\infty$ for every $n$.  Then
$
\bigcup_n {G_n}
$
is hyperfinite.
\end{theorem}

This is an immediate specialization of \cite[Theorem 1.10]{CJMST}.

Let $X$ be a standard Borel space and let
$
f_1,\ldots,f_d:X\longrightarrow X
$
be commuting Borel functions.  For $a=(a_1,\ldots,a_d)\in\N^d$, we write
\[
\PhiA_a=f_1^{a_1}\cdots f_d^{a_d},
\qquad
\abs{a}=a_1+\cdots+a_d.
\]

We say that a subset $A$ is \emph{forward invariant} along $f_i$ if $f_i(A)\subseteq A$ and \emph{backward invariant} along $f_i$ if $f_i^{-1}(A)\subseteq A$. If a set is both forward and backward invariant along any $f_i$, then it is invariant in the equivalence relation. We say that a subset $A$ is \emph{forward recurrent} if for any vertex $x\notin A$ there is a directed path from $x$ to some vertex in $A$, see also \cite{Higgins}. The \emph{free part} is
\[
F
=
\{x\in X:\PhiA_a(x)=\PhiA_b(x)\Longrightarrow a=b\}.
\]
For an extended metric $\rho$, write
\[
E_\rho=\{(x,y)\in X^2:\rho(x,y)<\infty\}.
\]

A \emph{commutative monoid} $(P,+)$ is a commutative semigroup with its addition operation both associative and commutative. $P$ is \emph{finitely generated} if there is a finite generating set. $P$ is \emph{cancellative} if $$\forall\ a,b,c\ a+c=b+c\Longrightarrow a=b.$$ Its \emph{Grothendieck group}, or \emph{group
completion}, is
\[
\operatorname{Gp}(P):=(P\times P)/{\sim},
\]
where
\[
(a,b)\sim(c,d)
\quad\Longleftrightarrow\quad
a+d=b+c.
\]
The equivalence class of \((a,b)\) is denoted by \(a-b\), and the group
operation is
\[
(a-b)+(c-d)=(a+c)-(b+d).
\]

There is a canonical monoid homomorphism
\[
\iota:P\longrightarrow \operatorname{Gp}(P),
\qquad
a\longmapsto a-0.
\]
Since \(P\) is cancellative, this map is injective.

Because \(P\) is finitely generated, \(\operatorname{Gp}(P)\) is a
finitely generated abelian group. Hence
$
\operatorname{Gp}(P)\cong \mathbb Z^r\oplus T,
$
where \(T\) is a finite abelian group. The \emph{rank} of \(P\) is defined to be
\[
\operatorname{rank}(P)
:=
\operatorname{rank}_{\mathbb Z}\operatorname{Gp}(P)
=
\dim_{\mathbb Q}\left(
\operatorname{Gp}(P)\otimes_{\mathbb Z}\mathbb Q
\right)
=r.
\]

\section{Proof of Theorem 1.1 and an upgrade}
\begin{theorem}
\label{hyp:free}
Let $X$ be a standard Borel space and $P$ be a finitely generated cancellative commutative monoid, and
let $P$ act Borelly on  $X$.  Suppose that
\begin{enumerate}[label=(\roman*)]
\item the action is free, meaning $p\cdot x=q\cdot x$ implies $p=q$;
\item a fixed finite generating set of $P$ acts by bounded-to-one
Borel functions.
\end{enumerate}
Then the graph generated by functions by the finite generating set has finite Borel
asymptotic dimension.
\end{theorem}

We remark that $E$ is induced by a finitely generated commutative monoid action, is equivalent to, $E$ is the connectness relation of graph generated by finitely many commuting functions. However, $E$ is induced by a finitely generated commutative monoid free action, is not equivalent to, $E$ is the connectness relation of the free part of graph generated by finitely many commuting functions.

The following proof of Theorem~1.1 is due to Naryshkin, Shinko, Weilacher and Yu, we include the argument for completeness, the idea is basically the proof of \cite[Lemma 8.3]{CJMST}.

\begin{lemma}[Forward recurrent marker lemma]\label{lem:forward-markers}
For every integer $m\geq1$ there is a Borel set $A\subseteq F$ such
that, for every $x\in F$,
\begin{enumerate}[label=(\roman*)]
\item if $a,b\in\N^d$, $a\neq b$, and
$\PhiA_a(x),\PhiA_b(x)\in A$, then
$
\lVert a-b\rVert_\infty>m;
$
\item there is $a\in\N^d$ with
$\lVert a\rVert_\infty\leq2m$ such that $\PhiA_a(x)\in A$.
\end{enumerate}
\end{lemma}

\begin{proof}
Since the functions are bounded-to-one, the graph
$G\upharpoonright F$ has bounded degree.  Let $G_m$ be the Borel graph
on $F$ in which two distinct points are adjacent when their
$G$-distance is at most $dm$.  The graph $G_m$ also has bounded
degree, so by \cite{KST1999} it admits a Borel proper $M$-coloring $c'\colon F\longrightarrow\{1,\ldots,M\}$ for some finite $M$.

We first describe the local greedy algorithm on $\mathbb Z^d$ that will be used later.  Let $D\subseteq\mathbb Z^d$, and suppose that
$\kappa\colon D\to\{1,\ldots,M\}$ has the property that
\[
0<\lVert u-v\rVert_\infty\leq m
\quad\Longrightarrow\quad
\kappa(u)\neq\kappa(v).
\]
Starting with $C^0=\varnothing$, define, for $1\leq i\leq M$,
\[
C^i
=
C^{i-1}
\cup
\left\{
u\in D:
\kappa(u)=i
\text{ and }
\dist_\infty(u,C^{i-1})>m
\right\}.
\]
The set $C^M$ is $m$-separated.  It is also $m$-covering in $D$:
if a point of color $i$ was not added at stage $i$, then it was
within distance $m$ of $C^{i-1}$.  Moreover, whether
$u\in C^i$ is determined by the colored
$\ell^\infty$-ball of radius $im$ about $u$.  Consequently, if
$
L=Mm,
$
there is a fixed finite rule $\mathcal G$ which decides whether
$u\in C^M$ from the restriction of $\kappa$ to
$u+[-L,L]^d$, whenever this box is contained in $D$.

Put
\[
R=2L,
\qquad
\mathbf R=(R,\ldots,R)\in\mathbb N^d,
\]
and define the shifted coloring
\[
c(x)=c'(\PhiA_{\mathbf R}(x)).
\]
This coloring is still $m$-separated on every forward orbit.  Indeed,
if $a,b\in\mathbb N^d$ are distinct and
$\lVert a-b\rVert_\infty\leq m$, then freeness implies that
$\PhiA_{\mathbf R+a}(x)$ and $\PhiA_{\mathbf R+b}(x)$ are distinct,
while
\[
d_G\bigl(\PhiA_{\mathbf R+a}(x),
          \PhiA_{\mathbf R+b}(x)\bigr)
\leq
\lVert a-b\rVert_1
\leq dm.
\]
Hence these two points have different $c'$-colors.

The shift also makes the coloring locally constant on fibers.  More
precisely, if $p,q\in\mathbb N^d$, $p\leq\mathbf R$
coordinatewise, and
$
\PhiA_p(y)=\PhiA_q(x),
$
then
\begin{equation}\label{eq:local-fiber-color}
c(y)
=
c'\bigl(\PhiA_{\mathbf R-p+q}(x)\bigr).
\end{equation}
Thus, inside an inverse neighborhood of radius at most $L$, the color
of a local inverse image depends only on its signed
$\mathbb Z^d$-coordinate, and not on which inverse image represents
that coordinate.

Now we run the greedy algorithm on $c$. Because the coloring is locally constant on fibers and the algorithm depends only on $[-L,L]^d$-neighborhood, we can see the graph as the Cayley graph of $\mathbb Z^d$ locally.

For each $x\in F$ and $v\in[-L,L]^d\cap\mathbb Z^d$, define its virtual
local color by
\[
c_x(v)
=
c'\bigl(\PhiA_{\mathbf R+v}(x)\bigr).
\]
The exponent vector $\mathbf R+v$ is nonnegative because $R\geq L$.
Formula \eqref{eq:local-fiber-color} says that this is exactly the
color seen at coordinate $v$ in any local inverse chart in which
that coordinate is represented.  We now define, uniformly,
\[
A
=
\left\{
x\in F:
\mathcal G\bigl((c_x(v))_{v\in[-L,L]^d}\bigr)=1
\right\}.
\]
This set is Borel, since its membership test uses only finitely many
Borel images of $x$.

It remains to verify the two required properties.  Fix $x\in F$ and
put
\[
D_R
=
\{u\in\mathbb Z^d:u_i\geq-R\text{ for every }i\}.
\]
Define
\[
\kappa_x(u)
=
c'\bigl(\PhiA_{\mathbf R+u}(x)\bigr),
\qquad u\in D_R,
\]
and run the preceding greedy algorithm on $(D_R,\kappa_x)$; denote
its output by $C_x$.  The color classes of $\kappa_x$ are
$m$-separated: if
$0<\lVert u-v\rVert_\infty\leq m$, then the corresponding points in
the forward orbit of $x$ are distinct by freeness and are at
$G$-distance at most $dm$.

For every $a\in\mathbb N^d$, the box
$a+[-L,L]^d$ is contained in $D_R$, and
\[
c_{\PhiA_a(x)}(v)
=
c'\bigl(\PhiA_{\mathbf R+a+v}(x)\bigr)
=
\kappa_x(a+v).
\]
By locality and translation invariance of the greedy rule,
\begin{equation}\label{eq:uniform-greedy}
\PhiA_a(x)\in A
\quad\Longleftrightarrow\quad
a\in C_x.
\end{equation}

If $\PhiA_a(x),\PhiA_b(x)\in A$, then
$a,b\in C_x$ by \eqref{eq:uniform-greedy}.  Since $C_x$ is
$m$-separated, this proves (i).

For (ii), in fact fix any $a\in\mathbb N^d$ and let
\[
\mathbf m=(m,\ldots,m).
\]
Since $C_x$ is $m$-covering in $D_R$, there is $v\in C_x$ such that
\[
\lVert v-(a+\mathbf m)\rVert_\infty\leq m.
\]
It follows coordinatewise that
\[
a_i\leq v_i\leq a_i+2m.
\]
In particular, $v\in\mathbb N^d$, and
\eqref{eq:uniform-greedy} gives $\PhiA_v(x)\in A$.  Taking $a=0$
proves (ii), while the same argument shows the stronger forward
$2m$-covering property from every point of the forward orbit.
\end{proof}

\begin{proof}[Proof of Theorem~1.1]
Restricted on the Borel set $F$, and write $\rho$ for the path
metric of $G\upharpoonright F$.  Fix $r>0$ and put
$m=\max\{1,\lceil r\rceil\}$.  Apply
Lemma~\ref{lem:forward-markers} at scale $m$, we have a Borel set
$A\subseteq F$.

For each $x\in F$, let $a(x)$ be the lexicographically least vector
$a\in\{0,\ldots,2m\}^d$ such that
$\PhiA_a(x)\in A$.  This is a Borel function.  Put
\[
\tau=(m,\ldots,m)
\]
and define the Borel map $\pi:F\to A$ by
\[
\pi(x)
=
\PhiA_{a(\PhiA_\tau(x))}(\PhiA_\tau(x)).
\]
We define $E$ to be an equivalence relation by
\[
x\,E\,y
\quad\Longleftrightarrow\quad
\pi(x)=\pi(y).
\]
For every $x$,
\[
\rho(x,\pi(x))
\leq
\abs{\tau}+\abs{a(\PhiA_\tau(x))}
\leq
dm+2dm=3dm.
\]
Thus every $E$-class has $\rho$-diameter at most $6dm$.

It remains to bound uniformly the number of $E$-classes meeting an
$r$-ball.  Fix $x\in F$ and let $y\in B_\rho(x,r)$.  By
Lemma~\ref{lem:common} (it is also true for free part), choose $p,q\in\N^d$ such that
\[
\abs{p}+\abs{q}\leq m
\quad\text{and}\quad
\PhiA_p(x)=\PhiA_q(y).
\]
Since every coordinate of $q$ is at most $m$, the vector
$\tau-q$ is nonnegative, and hence
\[
\PhiA_\tau(y)
=
\PhiA_{\tau-q}(\PhiA_q(y))
=
\PhiA_{\tau-q+p}(x).
\]
Write $b=a(\PhiA_\tau(y))$, we have
\[
\pi(y)=\PhiA_{\tau-q+p+b}(x).
\]
Every coordinate of $\tau-q+p+b$ lies between $0$ and $4m$.
Moreover, $\pi(y)\in A$.  Therefore all values of $\pi$ on
$B_\rho(x,r)$ lie in
\[
A\cap
\{\PhiA_v(x):v\in\{0,\ldots,4m\}^d\}.
\]

Partition $\{0,\ldots,4m\}$ into five intervals, each of diameter at
most $m$, and take the resulting product partition of
$\{0,\ldots,4m\}^d$ into at most $5^d$ boxes.  By
Lemma~\ref{lem:forward-markers}(i), two distinct exponent vectors
corresponding to points of $A$ cannot lie in the same box.  Hence
$\pi$ takes at most $5^d$ values on $B_\rho(x,r)$.  Equivalently, the
$r$-ball meets at most $5^d$ many $E$-classes.

The relation $E$ therefore witnesses
\[
\asdimB(G\upharpoonright F)\leq 5^d-1<\infty.
\]

\end{proof}

Now we prove Theorem 3.1.

\begin{lemma}\label{lem:coarse-transfer}
Let $\rho$ and $\sigma$ be the Borel extended metrics on a standard Borel
space $X$, and suppose their finite-distance relations are countable.
Assume that:
\begin{enumerate}[label=(\roman*)]
\item $E_\sigma\subseteq E_\rho$, and every $E_\rho$-class contains at
most $M$ many $E_\sigma$-classes;
\item there is $A<\infty$ such that
$\rho(x,y)\leq A\sigma(x,y)$ whenever $\sigma(x,y)<\infty$;
\item for every $R>0$ there is $\theta(R)<\infty$ such that
\[
\rho(x,y)\leq R\text{ and }\sigma(x,y)<\infty
\quad\Longrightarrow\quad
\sigma(x,y)\leq\theta(R).
\]
\end{enumerate}
If $\asdimB(X,\sigma)\leq n$, then
\[
\asdimB(X,\rho)\leq M(n+1)-1.
\]
\end{lemma}

\begin{proof}
Fix $r>0$, put $s=\max\{1,\theta(2r)\}$, and apply the definition of Borel
asymptotic dimension for $\sigma$ at scale $s$.  Thus there is a Borel
equivalence relation $F$ whose classes have $\sigma$-diameter at most some
$D<\infty$, and every $\sigma$-ball of radius $s$ meets at most $n+1$ many
$F$-classes.  By (ii), every $F$-class has
$\rho$-diameter at most $AD$.

A $\rho$-ball $B_\rho(x,r)$ meets at most $M$ many
$E_\sigma$-classes.  For each such $E_\sigma$-class $C$ meeting the ball, choose 
$y_C\in C\cap B_\rho(x,r)$ (it does not have to a Borel choice).  If
$z\in C\cap B_\rho(x,r)$, then $\rho(y_C,z)\leq 2r$ and
$\sigma(y_C,z)<\infty$, so (iii) gives
$\sigma(y_C,z)\leq\theta(2r)$.  Consequently
$C\cap B_\rho(x,r)$ meets at most $n+1$ many $F$-classes.  Hence the
whole $\rho$-ball meets at most $M(n+1)$ many $F$-classes.  The same
$F$ therefore witnesses the bound of Borel asymptotic dimension for $\rho$.
\end{proof}

\begin{lemma}\label{lem:cofinal-submonoid}
Let $P$ be a finitely generated cancellative commutative monoid, let
$K=\operatorname{Gp}(P)$ be its Grothendieck group, and suppose that
$r=\operatorname{rank}(K)>0$.  Then there are elements
$q_1,\ldots,q_r\in P$ such that:
\begin{enumerate}[label=(\roman*)]
\item $q_1,\ldots,q_r$ are $\mathbb Z$-linearly independent;
\item $L=\mathbb Zq_1+\cdots+\mathbb Zq_r$ has finite index in $K$;
\item putting $Q=\mathbb Nq_1+\cdots+\mathbb Nq_r$, for every finite
set $A\subseteq P$ contained in one coset of $L$ there is $c\in P$
such that $A+c\subseteq Q$.
\end{enumerate}
In particular, $Q$ is a free commutative monoid of rank $r$.
\end{lemma}

\begin{proof}
Cancellativity identifies $P$ with a submonoid of $K$.  Since $P$
generates $K$ as a group, we may choose
$q_1,\ldots,q_r\in P$ whose images form a basis of the rational vector
space $K\otimes_\mathbb Z\mathbb Q$.  They are $\mathbb Z$-linearly
independent, and $L=\sum_i\mathbb Zq_i$ has finite index in $K$.

The image of $P$ in the finite group $K/L$ is a submonoid.  Every
submonoid of a finite group is a subgroup, and this subgroup generates
$K/L$; hence the image is all of $K/L$.  Let $A\subseteq P$ be finite
and contained in the coset $\gamma\in K/L$.  Choose $c_0\in P$ whose
image is $-\gamma$.  For every $a\in A$ there are unique integers
$z_i(a)$ such that
\[
a+c_0=\sum_{i=1}^r z_i(a)q_i.
\]
Choose $N$ so large that $z_i(a)+N\geq0$ for every $a\in A$ and every
$i$, and set $c=c_0+N(q_1+\cdots+q_r)$.  Then $c\in P$ and
\[
a+c=\sum_{i=1}^r\bigl(z_i(a)+N\bigr)q_i\in Q
\]
for every $a\in A$, proving (iii).
\end{proof}

For a finite generating set $S$ of a commutative monoid $P$, let
$\ell_S(p)$ denote the least number of elements of $S$ whose sum is
$p$ (with $\ell_S(0)=0$).

\begin{lemma}\label{lem:action-metrics}
Let a finitely generated cancellative commutative monoid $P$ act
freely on a set $X$, let $S$ be a finite generating set, and let
$\Gamma$ be the graph generated by the maps
$x\mapsto s\cdot x$, $s\in S$.  Let $q_1,\ldots,q_r$, $L$, and $Q$ be
as in Lemma~\ref{lem:cofinal-submonoid}, and let $H$ be the graph generated by $x\mapsto q_i\cdot x$, $1\leq i\leq r$.  Then:
\begin{enumerate}[label=(\roman*)]
\item every $H$-component is contained in a $\Gamma$-component, and
every $\Gamma$-component contains at most $[K:L]$ many
$H$-components;
\item there is $A<\infty$ such that
$d_\Gamma(x,y)\leq A d_H(x,y)$ whenever $d_H(x,y)<\infty$;
\item for every $R>0$ there is $\theta(R)<\infty$ such that
\[
d_\Gamma(x,y)\leq R\text{ and }d_H(x,y)<\infty
\quad\Longrightarrow\quad
d_H(x,y)\leq\theta(R).
\]
\end{enumerate}
Here $K=\operatorname{Gp}(P)$ denotes the Grothendieck group of $P$.
\end{lemma}

\begin{proof}
First observe that whenever $x$ and $y$ lie in the same
$\Gamma$-component, there are $p,q\in P$ such that
\begin{equation}\label{eq:general-common-future}
p\cdot x=q\cdot y.
\end{equation}
Indeed, this follows by induction along a path.  The induction step is
the identity
\[
p\cdot x=q\cdot y,\quad u\cdot y=v\cdot z
\quad\Longrightarrow\quad
(p+u)\cdot x=(q+v)\cdot z.
\]
Moreover, if the path has length at most $m$, the witnesses may be
chosen with $\ell_S(p)+\ell_S(q)\leq m$.

Define
\[
\delta(x,y)=p-q\in K
\]
using any witnesses in \eqref{eq:general-common-future}.  This is
well-defined.  Namely, if also $p'\cdot x=q'\cdot y$, then
\[
(q+p')\cdot y=(p+p')\cdot x=(q'+p)\cdot y,
\]
and freeness at $y$ gives $q+p'=q'+p$, hence $p-q=p'-q'$ in $K$.
The same calculation, or the displayed induction step, shows that
\begin{equation}\label{eq:cocycle-additive}
\delta(x,z)=\delta(x,y)+\delta(y,z)
\end{equation}
whenever the three points lie in one $\Gamma$-component.

We claim that two points in one $\Gamma$-component are
$H$-connected exactly when their $\delta$-value belongs to $L$.  One
direction follows from \eqref{eq:cocycle-additive}, since an oriented
$H$-edge has $\delta$-value $q_i$ or $-q_i$.  Conversely, suppose that
$p\cdot x=q\cdot y$ and $p-q\in L$.  Then $p$ and $q$ lie in the same
coset of $L$.  Lemma~\ref{lem:cofinal-submonoid}, applied to
$\{p,q\}$, gives $c\in P$ such that $p+c,q+c\in Q$.  Applying $c$ to
the equality gives
\[
(p+c)\cdot x=(q+c)\cdot y.
\]
Both sides are reached using only the $q_i$, so the undirected graph
$H$ connects $x$ to $y$.

Fixing one point $x$ in a $\Gamma$-component, the map which assigns to
an $H$-component the coset $\delta(x,y)+L$ of any point $y$ in that
component is therefore well-defined and injective.  This proves the
component bound in (i).  The containment of components follows also
from the fact that every $q_i$ is a sum of elements of $S$.

Set
\[
A=\max_{1\leq i\leq r}\ell_S(q_i).
\]
Every $H$-edge can be replaced by a $\Gamma$-path of length at most
$A$, which proves (ii).

Finally fix $R>0$, put $m=\lceil R\rceil$, and consider the finite set
\[
\mathcal W_m=
\{(p,q)\in P^2:\ell_S(p)+\ell_S(q)\leq m,\ p-q\in L\}.
\]
For each $(p,q)\in\mathcal W_m$, use
Lemma~\ref{lem:cofinal-submonoid} to choose $c_{p,q}\in P$ with
$p+c_{p,q},q+c_{p,q}\in Q$, and write
\[
p+c_{p,q}=\sum_i a_i(p,q)q_i,
\qquad
q+c_{p,q}=\sum_i b_i(p,q)q_i
\]
with nonnegative integer coefficients.  Let
\[
\theta(R)=
\max_{(p,q)\in\mathcal W_m}
\sum_i\bigl(a_i(p,q)+b_i(p,q)\bigr).
\]
If $d_\Gamma(x,y)\leq R$ and $x,y$ are $H$-connected, the first
paragraph supplies $(p,q)\in\mathcal W_m$ with $p\cdot x=q\cdot y$.
After applying $c_{p,q}$, the displayed $Q$-coordinates give an
$H$-path from $x$ to $y$ of length at most $\theta(R)$.  This proves
(iii).
\end{proof}

\begin{proof}[Proof of Theorem~3.1]
Let $S$ be the fixed finite generating
set, let $K$ be the Grothendieck group of $P$, and let $\Gamma$ be the
graph in the statement.

First suppose that $\operatorname{rank}(K)=0$.  Then $K$ is finite.
The image of $P$ in $K$ is a submonoid of a finite group, hence a
subgroup; since it generates $K$, it equals $K$.  Thus $P$ is a finite
group.  Every map in the action is then a bijection, and every
$\Gamma$-component is a $P$-orbit of cardinality at most $\lvert
P\rvert$.  The connectedness relation of $\Gamma$ is a finite Borel
equivalence relation, so
$\asdimB(\Gamma)=0$.

Now suppose that $r=\operatorname{rank}(K)>0$.  Choose
$q_1,\ldots,q_r\in P$ and put
\[
L=\mathbb Zq_1+\cdots+\mathbb Zq_r,
\qquad
Q=\mathbb Nq_1+\cdots+\mathbb Nq_r
\]
as in Lemma~\ref{lem:cofinal-submonoid}.  Let $H$ be the undirected
Borel graph generated by the maps
\[
x\longmapsto q_i\cdot x,
\qquad 1\leq i\leq r.
\]
Each of these maps is a composition of maps coming from $S$, and is
therefore Borel and bounded-to-one.  They commute.  Moreover, the
resulting $\mathbb N^r$-action is free: if
\[
\Bigl(\sum_i a_iq_i\Bigr)\cdot x
=
\Bigl(\sum_i b_iq_i\Bigr)\cdot x,
\]
then freeness of the $P$-action gives
$\sum_i a_iq_i=\sum_i b_iq_i$ in $P$, and the
$\mathbb Z$-linear independence of the $q_i$ gives $a_i=b_i$ for all
$i$.  Thus all of $X$ is the free part for these $r$ commuting maps.
Theorem~1.1 therefore gives
\[
\asdimB(H)<\infty.
\]

Both $\Gamma$ and $H$ are locally finite Borel graphs, because their
finite generating families consist of bounded-to-one Borel maps.
Lemma~\ref{lem:action-metrics} says that their path metrics satisfy all
three hypotheses of Lemma~\ref{lem:coarse-transfer}, with
$M=[K:L]$.  Applying that lemma to
$\rho=d_\Gamma$ and $\sigma=d_H$ now yields
\[
\asdimB(\Gamma)<\infty,
\]
as required.
\end{proof}
\section{Proof of Theorem 1.2}

\begin{lemma}[Common future lemma]\label{lem:common}
For all $x,y\in X$,
\[
d_G(x,y)
=
\min\bigl\{\abs{a}+\abs{b}:
a,b\in\N^d,\ \PhiA_a(x)=\PhiA_b(y)\bigr\},
\]
with the minimum interpreted as $\infty$ if the set is empty.
\end{lemma}
\begin{proof}
The equation $\PhiA_a(x)=\PhiA_b(y)$ gives a path of length at most
$\abs{a}+\abs{b}$: move forward from $x$ to the common future, and
then move from the common future to $y$ backwards. So
\[
d_G(x,y)
\leq
\min\bigl\{\abs{a}+\abs{b}:
a,b\in\N^d,\ \PhiA_a(x)=\PhiA_b(y)\bigr\}.
\]

Conversely, suppose $x=x_0,x_1,\cdots,x_n=y$ is a path in $G$ and 
$$\PhiA_{a_i}(x_i)=\PhiA_{b_i}(x_{i+1})\quad\text{where }(|a_i|,|b_i|)=(0,1)\text{ or }(1,0)$$
Note that if
\[
\PhiA_a(x)=\PhiA_b(y)
\quad\text{and}\quad
\PhiA_c(y)=\PhiA_e(z),
\]
then commutativity gives
\begin{equation}
\PhiA_{a+c}(x)
=
\PhiA_{b+c}(y)
=
\PhiA_{b+e}(z).\label{eq:commute}
\end{equation}
So
$$\PhiA_{\Sigma a_i}(x)=\PhiA_{\Sigma b_i}(y).$$
This shows
\[
d_G(x,y)
\geq
\min\bigl\{\abs{a}+\abs{b}:
a,b\in\N^d,\ \PhiA_a(x)=\PhiA_b(y)\bigr\}.
\]
\end{proof}
We remark that there could be $f_i(x)=x$ for some vertex, in this case $x$ is the fixed point of $f_i$, one can verify that it does not affect the above proof.
\begin{lemma}\label{cor:intrinsic}
If $A\subseteq X$ is forward invariant along any $f_i$, then for
$x,y\in A$,
\[
d_{G\upharpoonright A}(x,y)=d_G(x,y).
\]
\end{lemma}

\begin{proof}
$d_{G\upharpoonright A}(x,y)\geq d_G(x,y)$ is trivial. Conversely, by Lemma \ref{lem:common} we choose $a,b$ such that $d_G(x,y)=|a|+|b|$ and $\PhiA_a(x)=\PhiA_b(y)$, since $A$ is forward invariant, both forward paths stay in $A$. After deleting the stationary steps, this gives a path in $G\upharpoonright A$ of length at most $|a|+|b|=d_G(x,y)$.
\end{proof}

\begin{lemma}\label{lem:subgroup}
We define 
\[
\Lambda(x)
=
\{a-b\in\Z^d:\PhiA_a(x)=\PhiA_b(x)
\text{ for some }a,b\in\N^d\}.
\]For every $x\in X$, the set $\Lambda(x)$ is a subgroup of $\Z^d$.
\end{lemma}

\begin{proof}
It is routine to check that it contains $0$ and is closed under addition and negation using formula (\ref{eq:commute}).
\end{proof}

\begin{lemma}\label{lem:invariant}
If $y=f_i(x)$, then
\[
\Lambda(y)=\Lambda(x).
\]
Consequently $x\mapsto\Lambda(x)$ is constant on every $G$-component.
\end{lemma}

\begin{proof}
The equality $\PhiA_a(x)=\PhiA_b(x)$ remains true after applying $f_i$, so
$\Lambda(x)\subseteq\Lambda(y)$.  Conversely, if
\[
\PhiA_a(y)=\PhiA_b(y),
\]
then
\[
\PhiA_{a+e_i}(x)=\PhiA_{b+e_i}(x),
\]
and the difference of the two vectors is again $a-b$.
Thus $\Lambda(y)\subseteq\Lambda(x)$.
\end{proof}

Hence, the sets
\[
X_H=\{x:\Lambda(x)=H\},
\qquad H\leq\Z^d,
\]
form a countable Borel disjoint partition.  The free part is
$X_{\{0\}}$, and the non-free part is the union of the $X_H$ with
$H\ne\{0\}$.

For a fixed $H$, the relation
\[
a\equiv_H b\quad\Longleftrightarrow\quad a-b\in H
\]
is a congruence on the free commutative monoid $\N^d$.

\begin{theorem}[R\'edei, see \cite{RedeiShort} and \cite{DiaconisSturmfels}]\label{lem:markov}
There is a finite set
\[
\mathcal M_H
=
\{(u_1,v_1),\ldots,(u_s,v_s)\}
\subseteq\N^d\times\N^d
\]
that generates $\equiv_H$ as a monoid congruence.  Explicitly, if
$a-b\in H$, then there is a finite sequence
\[
a=a_0,a_1,\ldots,a_\ell=b
\]
such that, for each $k$, there are $j\leq s$ and $w\in\N^d$ with
\[
\{a_k,a_{k+1}\}
=
\{w+u_j,w+v_j\}.
\]
\end{theorem}
R\'edei's theorem says that every congruence on a finitely generated
commutative semigroup is finitely generated, see \cite{RedeiShort}.  In the language of algebraic statistics,
$\mathcal M_H$ is a finite Markov basis for the lattice $H$; finite
existence also follows from the Hilbert basis theorem applied to the
corresponding binomial ideal, see \cite{DiaconisSturmfels}.

Fix such a basis. We define the \emph{core} by
\[
C_H
=
\{x\in X_H:\PhiA_{u_j}(x)=\PhiA_{v_j}(x)
\text{ for every }j\leq s\}.
\]

We will show the graph restricted on the core $C_H$ is of finite Borel asymptotic dimension and then it will give hyperfiniteness on $X_H$.

\begin{lemma}\label{lem:coreproperties}
The set $C_H$ is Borel and forward invariant.  For every $x\in C_H$
and every $a,b\in\N^d$,
\[
\PhiA_a(x)=\PhiA_b(x)
\quad\Longleftrightarrow\quad
a-b\in H.
\]
\end{lemma}

\begin{proof}
Borelness and forward invariance are obvious.  If $a-b\in H$, use
the chain from Lemma \ref{lem:markov}. $\PhiA_{a_i}(x)=\PhiA_{a_{i+1}}(x)$, by commutativity and the definition of $C_H$.
Thus $\PhiA_a(x)=\PhiA_b(x)$.

Conversely, if $\PhiA_a(x)=\PhiA_b(x)$, then
$a-b\in\Lambda(x)=H$ because $x\in X_H$.
\end{proof}

Let
\[
q_H:\N^d\longrightarrow \Z^d/H
\]
be the quotient map $q_H(a)=[a]_H$ and put
$
P_H=q_H(\N^d).
$ It is a finitely generated cancellative commutative monoid.

The following proposition is the only use of Theorem 3.1.

\begin{proposition}[Finite-dimensional core]
\label{prop:coreasdim}
For every subgroup $H\leq\Z^d$,
\[
\asdimB(G\!\upharpoonright C_H)<\infty.
\]
\end{proposition}

\begin{proof}
Lemma \ref{lem:coreproperties} shows that the action on $C_H$ by $P_H$ given by $[a]_H\cdot x=\PhiA_{a}(x)$ does not depend on the choice of $a$.  It is free as a $P_H$ action, indeed, if
$q_H(a)\ne q_H(b)$, then $a-b\notin H$, so
$\PhiA_a(x)\ne\PhiA_b(x)$.  The generators $q_H(e_i)$ act by the
restrictions of the original bounded-to-one maps. Theorem 3.1 applies.
\end{proof}

Now we follow the idea of \cite[Lemma 8.3]{CJMST}.

\begin{lemma}\label{lem:shiftrelation}
Let $x\in X_H$, and let $u,v\in\N^d$ satisfy $u-v\in H$.  Then there
is $c\in\N^d$ such that
\[
\PhiA_{u+c}(x)=\PhiA_{v+c}(x).
\]
\end{lemma}

\begin{proof}
Since $u-v\in H=\Lambda(x)$, choose $a,b\in\N^d$ with
\[
a-b=u-v
\quad\text{and}\quad
\PhiA_a(x)=\PhiA_b(x).
\]
Let
$
z=a-u=b-v\in\Z^d.
$
Choose $r\in\N^d$ so that $z+r\in\N^d$ and put $c=z+r$.  Applying
$\PhiA_r$ to the equality $\PhiA_a(x)=\PhiA_b(x)$ gives
\[
\PhiA_{u+c}(x)=\PhiA_{v+c}(x).
\]
\end{proof}

\begin{proposition}\label{prop:reachcore}
$C_H$ is forward recurrent in $X_H$.
\end{proposition}

\begin{proof}
For each Markov basis pair $(u_j,v_j)$, choose $c_j\in\mathbb N^d$ by
Lemma \ref{lem:shiftrelation}
\[
\PhiA_{u_j+c_j}(x)=\PhiA_{v_j+c_j}(x).
\]
Let $c$ be the coordinatewise maximum
of $c_1,\ldots,c_s$. After applying further forward functions, we have
\[
\PhiA_{u_j+c}(x)=\PhiA_{v_j+c}(x)
\]
for every $j$.  Hence $\PhiA_c(x)\in C_H$.
\end{proof}

For $m\geq0$, we define the bounded-depth layer
\[
Y_{H,m}
=
\bigl\{
x\in X_H:
\PhiA_c(x)\in C_H
\text{ for some }c\in\N^d,\ \abs{c}\leq m
\bigr\}.
\]

\begin{lemma}\label{lem:Yproperties}
The sets $Y_{H,m}$ are Borel and forward invariant along any $f_i$, and
\[
Y_{H,0}\subseteq Y_{H,1}\subseteq\cdots,
\qquad
X_H=\bigcup_m Y_{H,m}.
\]
\end{lemma}

\begin{proof}
If $\PhiA_c(x)\in C_H$, then
\[
\PhiA_c(f_i(x))=f_i(\PhiA_c(x))\in C_H
\]
because $C_H$ is forward invariant.  Thus $Y_{H,m}$ is forward
invariant.  The exhaustion follows from
Proposition \ref{prop:reachcore}.
\end{proof}

For $x\in Y_{H,m}$, let $c_m(x)$ be the lexicographically least
$c\in\N^d$ with $\abs{c}\leq m$ and $\PhiA_c(x)\in C_H$, and put
\[
\pi_m(x)=\PhiA_{c_m(x)}(x).
\]
Then $\pi_m:Y_{H,m}\to C_H$ is Borel and
\[
d_G(x,\pi_m(x))\leq m.
\]

\begin{lemma}\label{lem:retraction}
If $x,y\in Y_{H,m}$, then
\[
d_{G\upharpoonright C_H}(\pi_m(x),\pi_m(y))
\leq
(2m+1)d_{G\upharpoonright Y_{H,m}}(x,y).
\]
\end{lemma}

\begin{proof}
We prove by induction on $d_{G\upharpoonright Y_{H,m}}(x,y)$. First suppose $x$ and $y$ are adjacent. There is a path in $G$
\[
\pi_m(x)--- x- y---\pi_m(y)
\]
of length at most $2m+1$, and its endpoints lie in the forward-invariant
set $C_H$. By Lemma \ref{cor:intrinsic}, the distance in $G$ is the same as in $C_H$.  

Now we can finish the lemma by induction along a path in
$Y_{H,m}$.
\end{proof}

\begin{proposition}\label{prop:Yasdim}
For every $H\leq\Z^d$ and every $m$,
\[
\asdimB(G\!\upharpoonright Y_{H,m})<\infty.
\]
\end{proposition}

\begin{proof}
Fix $r>0$.  Choose on $C_H$ a uniformly bounded Borel equivalence
relation $E$ witnessing finite Borel asymptotic dimension at scale
$(2m+1)r$. We define an equivalence relation $\widetilde E$:
\[
x\,\widetilde E\,y
\quad\Longleftrightarrow\quad
\pi_m(x)\,E\,\pi_m(y).
\]
Lemma \ref{lem:retraction} shows that an $r$-ball in $Y_{H,m}$ maps
into a $(2m+1)r$-ball in $C_H$, so it meets no more
$\widetilde E$-classes than a $(2m+1)r$-ball meets
$E$-classes.

If the $E$-classes have $C_H$-diameter at most $S$, then a
$\widetilde E$-class has $G$-diameter at most
$
m+S+m.
$ And $G$-distance is the same as $d_{G\upharpoonright Y_{H,m}}$-distance, so $\widetilde E$ is uniformly bounded.
Thus $\widetilde E$ witnesses finite Borel asymptotic dimension.
\end{proof}

The rest of the proof is routine. By the increasing union theorem \ref{thm:union}, we have $G\upharpoonright X_H$ is hyperfinite, and $G\upharpoonright X_H$ is a countable disjoint union of $G$, so $G$ is hyperfinite.

\begin{acknowledgments}
    The author would like to thank Wei Dai and Cecelia Higgins for drawing attention to this question. The author would like to thank Petr Naryshkin for many helpful discussions.
\end{acknowledgments}


\begin{thebibliography}{99}

\bibitem{OpenProblems}
G.~Barmpalias, N.~Bazhenov, C.~T. Chong, W.~Dai, S.~Gao,
J.~L. Goh, J.~He, K.~M.~S. Ng, A.~Nies, T.~Slaman,
R.~Thornton, W.~Wang, J.~Yu, and L.~Yu,
\newblock Open problems in computability theory and descriptive set theory,
\newblock preprint (2025).

\bibitem{BY}
A.~Bernshteyn and J.~Yu,
\newblock Large-scale geometry of Borel graphs of polynomial growth,
\newblock \emph{Advances in Mathematics} \textbf{473} (2025), 110290.

\bibitem{CJMST}
C.~Conley, S.~Jackson, A.~Marks, B.~Seward and R.~Tucker-Drob,
\emph{Borel asymptotic dimension and hyperfinite equivalence relations},
Duke Math. J. \textbf{172} (2023), no.~16, 3175--3226.

\bibitem{DiaconisSturmfels}
P.~Diaconis and B.~Sturmfels,
\newblock Algebraic algorithms for sampling from conditional distributions,
\newblock \emph{Ann. Statist.} \textbf{26} (1998), no.~1, 363--397.

\bibitem{GJ2015}
S. Gao and S. Jackson,
\emph{Countable abelian group actions and hyperfinite equivalence relations},
Invent. Math. \textbf{201} (2015), no. 1, 309--383.

\bibitem{GH}
J.~Greb\'ik and C.~Higgins,
\newblock Complexity of finite Borel asymptotic dimension,
\newblock \emph{Forum Math. Sigma} \textbf{14} (2026), e31.

\bibitem{RedeiShort}
P.~Grillet,
\newblock A short proof of R\'edei's theorem,
\newblock \emph{Semigroup Forum} \textbf{46} (1993), 126--127.

\bibitem{Higgins} C. Higgins, \textit{A note on forward-recurrent sets with bounded gaps}, manuscript available at {\tt https://sites.google.com/view/cecelia-higgins/home}, 2023.

\bibitem{KST1999}
A. Kechris, S. Solecki and S. Todorcevic,
\emph{Borel chromatic numbers},
Adv. Math. \textbf{141} (1999), no. 1, 1--44.

\bibitem{NSWY}
P.~Naryshkin, F.~Shinko, F.~Weilacher and J.~Yu
\newblock Hyperfiniteness of bounded-to-one actions of commutative monoids, in preparation

\bibitem{SchneiderSeward2024}
S.~Schneider and B.~Seward,
\newblock Locally nilpotent groups and hyperfinite equivalence relations,
\newblock \emph{Mathematical Research Letters}
  \textbf{31} (2024), no.~2, 511--578.


\end{thebibliography}
\end{document}